\documentclass[12pt]{amsart}
\usepackage{amssymb}
\newtheorem{theorem}{Theorem}[section]

\newtheorem{corollary}[theorem]{Corollary}
\newtheorem{example}[theorem]{Example}
\newtheorem{question}{Question}[section]
\newtheorem{definition}[theorem]{Definition}

\begin{document}

\title[Selectively highly divergent spaces] {On some questions on selectively
highly divergent spaces}  
\author[A. Bella] {Angelo Bella}

\address{Dipartimento di Matematica e Informatica, Universit\` a di Catania, viale A. Doria
6, 95125 Catania, Italy}
\email{bella@dmi.unict.it}

\author{Santi Spadaro}\address{Dipartimento di Ingegneria,
Universit\` a di Palermo, Viale delle Scienze, Ed. 8, 90128, Palermo, Italy}
\email{santidomenico.spadaro@unipa.it, santidspadaro@gmail.com}


\begin{abstract} A  topological space $X$ is selectively highly
divergent (SHD) if for every sequence of non-empty open sets
$\{U_n: n\in \omega \}$ of $X$, we can find points $x_n\in U_n$, for every $n <\omega$
such that the sequence $\{x_n: n\in\omega\} $ has no
convergent subsequences. In this note we answer four questions related to this notion that were asked in ArXiv:2307.11992.

\end{abstract} 
\maketitle 
\bigskip

\section{Introduction} 
In this note  we consider  a class of spaces recently studied in
\cite{messico}.

 \begin{definition} \label{def} A
topological space $X$ is \textit{selectively highly divergent}
(SHD from here for short) if for every sequence of non-empty open
subsets $\{U_n: n<\omega  \}$ of $X$, we can find $x_n\in
U_n$ such that the sequence $\{x_n: n<\omega\}$  has no
convergent subsequence. \end{definition}

Clearly,  if a topological space $X$ has a point of
countable character, then it cannot be SHD, in particular   no
metrizable space is SHD.

Nice examples of SHD spaces are the compact Hausdorff space
$\omega^*=\beta \omega\setminus \omega$ and  the countable
regular maximal space $M$  described in \cite{vD1}.  These spaces
however are strictly stronger than  SHD because they do not
contain non-trivial convergent sequences.  In general, a
selectively higly divergent space may have  plenty of convergent
sequences:  a compact Hausdorff space of this kind is
$\omega^*\times I$, while a countable regular one is $M\times
\mathbb Q$.
\smallskip

The property of being selectively highly divergent is much stronger than being not sequentially compact. An easy example of non-sequentially compact space which is not SHD is the space $Z=\omega^*\oplus I$. Note that the space $Z$ has an open subset which is sequentially compact, and one may suspect that a space having no non-empty open sequentially
compact subspace should be SHD, but this is not the case.

\begin{example} \label{ex}  A compact Hausdorff  space with no
non-empty sequentially compact subspace which is not SHD.
\end{example}
\begin{proof} Let 
$X=(\omega^*\times \omega)\cup \{p\}$, where $\omega^*\times
\omega$   with the product topology is an open subspace of $X$,
while a local base at $p$ is the collection $\{(\omega^*\times
[n,\omega[)\cup \{p\}:n<\omega\}$. \end{proof}

If in Definition \ref{def} we consider only  constant  sequences
of open sets, i.e. $U_n=U$ for each $n<\omega$, then we see that
a SHD space has the property that  every non-empty open set
contains a sequence with no subsequences converging in $X$. We
may call a space with this property \textit{highly divergent}
(HD for short). Using this terminology, Example \ref{ex} provides an example of a compact Hausdorff HD space which is not SHD.

\smallskip

In \cite{messico}  the authors formulated various questions about selectively highly divergent spaces. In our paper we will focus on four of them.  

\begin{question} \label{q1} \cite[Question 2]{messico} 
Is it true that if $\kappa$ is an uncountable cardinal, then
$X=\{0,1 \}^{\kappa}$ is a SHD space?
\end{question}  

\begin{question}\label{q2}  \cite[Question 4]{messico} 
If $X$ is Tychonoff, non-compact and SHD, does it hold that
$\beta X$ is SHD? 
\end{question}

\begin{question} \label{q3} \cite[Question 5]{messico} Is the SHD
property dense hereditary? \end{question}

Given a space $X$, let $\mathcal{F}[X]$ denote the Pixley-Roy hyperspace of $X$.

\begin{question} \label{q4} \cite[Question 7]{messico}
Is $\mathcal{F}[X]$ SHD whenever $X$ is SHD and $T_1$?
\end{question}

In the present note, we give a complete answer to Questions \ref{q1}, \ref{q3} and \ref{q4} and a
partial positive answer to Question \ref{q2}.

All spaces are assumed to be $T_1$. For undefined notions, we refer the reader to \cite{En} and \cite{Ku}.

\section{The main results}

We begin by presenting a complete answer to Question \ref{q1}. 

Recall that a collection $\mathcal S$ of subsets of $\omega$ is a
splitting family if for every infinite subset $A\subseteq \omega$
there is an element $S\in \mathcal S$ satisfying $|S\cap
A|=|A\setminus S|=\omega$. The smallest cardinality of a
splitting family on $\omega$ is the splitting number $\mathfrak
s$. It turns out that $\omega_1\le \mathfrak s\le \mathfrak
c$.

\begin{theorem} \label{t1} 
The space $2^\kappa $ is   selectively
highly divergent if and only if $\kappa \ge \mathfrak s$.
\end{theorem}
\begin{proof}  If $\kappa <\mathfrak s$, then $2^\kappa $ is
sequentially compact (see \cite{vD2}, Theorem 6.1). So, if
$2^\kappa $
is SHD, then we should have $\kappa \ge \mathfrak s$. To complete
the
proof, we need to show that  $\kappa \ge \mathfrak s$ implies
that $2^\kappa $ is SHD.  Since $2^\kappa $ is homeomorphic to
$2^\mathfrak s\times 2^\kappa $, taking into account that  any
product having a SHD factor is SHD (see Theorem 1 in
\cite{messico}),   it 
suffices to prove that $2^\mathfrak  s$ is selectively
highly divergent.

 Let $\mathcal S$ be a splitting family on
$\omega$
of size $\mathfrak  s$ and fix an indexing $\mathcal S=\{S_\alpha
:\alpha
<\mathfrak  s\}$ in such a way that every element of $\mathcal
S$ appears
in the list $\mathfrak  s$-many times.

Recall that a base for the topology of $2^\kappa $ consists of
the sets  $[\sigma ]$, where $\sigma\in Fin(\kappa ,2)$ 
is a
partial function  whose domain is a finite subset of $\kappa $
and $[\sigma ]=\{x\in 2^\kappa :\sigma \subseteq x\}$. 
Let $\{U_n:n<\omega\}$ be a family of non-empty open subsets of
$2^\mathfrak  s$ and for each $n$ choose a partial function
$\sigma_n:\mathfrak s \to 2$ such that $[\sigma_n]\subseteq U_n$.

For each $n$ let  $x_n\in 2^\mathfrak  s$ be the point defined as
follows. If $\alpha \in dom(\sigma_n)$, then let $x_n(\alpha
)=\sigma_n(\alpha )$; if $\alpha \in \mathfrak  s\setminus
dom(\sigma_n)$, then let $x_n(\alpha )=1$ when $n\in S_\alpha $
and $x_n(\alpha )=0$ when $n\notin S_\alpha $. Of course, we have
$x_n\in [\sigma_n]\subseteq U_n$.

We claim that the sequence $\{x_n:n<\omega\}$ does not have
convergent subsequences. Assume by contradiction that the
subsequence $\{x_n:n\in A\}$ converges to a point $p$. Since the
family $\mathcal S$ is splitting, there exists $S\in \mathcal S$
such
that $|A\cap S|=|A\setminus S|=\omega$. Since   the set $\bigcup
\{dom(\sigma_n):n<\omega\}$ is countable and $S$ appears in the
list $\{S_\alpha :\alpha <\mathfrak s\}$  $\mathfrak  s$-many 
times,  we
may find $\gamma  \in \mathfrak  s\setminus \bigcup
\{dom(\sigma_n):n\in \omega\}$ such that $S_\gamma  =S$. Now,
since  the sequence $\{x_n:n\in A\cap S\}$ converges to $p$ and
$x_n(\gamma  )=1$ for each $n\in A\cap S$, we must have $p(\gamma
)=1$. But even the sequence $\{x_n:n\in A\setminus S\}$ converges
to $p$ and hence we must also have $p(\gamma  )=0$. As this is a
contradiction, the proof is complete. \end{proof}

Theorem \ref{t1} will help us answer Question \ref{q3} in the negative.   

\begin{example} \label{denseher}  A compact Hausdorff SHD space
with a dense subspace which is not SHD. \end{example}
\begin {proof}  Let $X=2^\mathfrak{c}$.   Theorem \ref{t1} sais
that $X$ is selectively highly divergent.   Let $Y$ be the $\Sigma$-product of $2^{\mathfrak{c}}$, that is $Y=\{x\in X: |x^{-
1}(1)| \leq \omega\}$, with the topology induced from $X$. Then $Y$ is a dense subset of $X$:
Since in $Y$ every countable set is contained in a copy of the Cantor set, we immediately see that $Y$ is sequentially compact. Thus, $Y$ is
a dense  subspace of $X$ which is not selectively higly
divergent. \end{proof}   

We now give a partial answer to Question \ref{q2}. Recall that  a set
$A\subseteq
X$ is $C^*$-embedded in $X$ if every bounded real valued
continuous function defined on $A$ can be continuously extented
to the whole of $X$.  The Tietze-Urysohn theorem  implies that every
closed subspace of a normal space is $C^*$-embedded.
\begin{theorem}  \label{t2} Let $X$ be a Tychonoff SHD space. If
every closed copy of the discrete space $\omega$ is $C^*$-
embedded, then $\beta X$ is SHD.\end{theorem}
\begin{proof} Let $\{U_n:n<\omega\}$ be a sequence of non-empty
open sets  of $\beta X$. Since $X$ is SHD,  we may pick points
$x_n\in U_n\cap X$ in such a way that  $\{x_n:n<\omega\}$ does
not have subsequences which are convergent in $X$. We claim  
that  $\{x_n:n<\omega\}$  does not have convergent subsequences
even in $\beta X$.

 Assume by contradiction that the sequence
$\{x_n:n\in A\}$ converges to a point $p\in \beta X$. Clearly, we
should have $p\in \beta X\setminus X$. But then, the set
$\{x_n:n\in A\}$ is closed  and discrete in $X$. Split $A$ in the
union of two infinite subsets $B$ and $C$ and define
$f:\{x_n:n\in A\}\to [0,1]$ by letting $f(x_n)=0$ in $n\in B$ and
$f(x_n)=1$ if $n\in C$. Since the set $\{x_n:n\in A\}$ is $C^*$-
embedded, we may continuously extend $f$ to a function  $f:X\to
[0,1]$.The
next  step is to extend $f$ to a continuous function $g:\beta
X\to [0,1]$.  Since $\{x_n:n\in A\}$ converges to $p$, we should
have $g(p)\in \overline {\{g(x_n):n\in B\}}=\overline
{\{f(x_n):n\in B\}}=\{0\}$, i. e. $g(p)=0$.   The same argument
shows that
$g(p)\in \overline {\{f(x_n):n\in C\}}=\{1\}$, i. e. $g(p)=1$. As
this is a
contradiction, the proof is complete.\end{proof}
We  may mention a couple of corollaries.
\begin{corollary} If $X$ is a normal SHD space, then  $\beta X$
is SHD.\end{corollary}
\begin{corollary} If $X$ is a countable Tychonoff  SHD space,
then $\beta X$ is SHD. \end{corollary}

So,  we see that $\beta M$ is
SHD.
\smallskip 

Example \ref{denseher} already shows that  the HD property is not
dense hereditary. We now describe another example which involves the
\v Cech-Stone compactification.  
 Let us consider the space  $\beta \mathbb Q$. It is
clear that $\mathbb Q$ is dense and far to be higly divergent. We
check that $\beta\mathbb Q$ is HD. To this end, let $U$ be a
non-empty open subset  of $\beta\mathbb Q$ and take a non-empty
open 
set $V$ such that $\overline V\subseteq U$.  The set $V\cap
\mathbb Q$ contains  a closed copy $A$ of the discrete space
$\omega$. Since $A$ is $C^*$-embedded in $\mathbb Q$, we have
that $\overline A\subseteq U$ is homeomorphic to $\beta\omega$
and so  every non-trivial sequence in $\overline A\subseteq U$
has no convergent subsequences in $\beta\mathbb Q$.

Notice that $\beta\mathbb Q $  is not SHD because it is
first countable at each point $q\in \mathbb Q$. So,  
$\beta\mathbb Q$ is another  compact Hausdorff  HD space which
is not SHD.   However, the   space $X$ given in Example \ref{ex}
is of different nature because every dense set $D$ of $X$ is
higly divergent.  To check this, let $U$ be a non-empty open set
in the subspace  $D$ and fix an open set $V$ of $X$ such that
$U=V\cap D$.  There is some $n\in \omega$ such that $V\cap
\omega^*\times \{n\}\ne \emptyset$ and so even $V\cap
\omega^*\times \{n\}\cap D=U\cap \omega^*\times\{n\}\ne \emptyset
$. Since the latter set is infinite, we may fix an infinite set
$\{x_n:n<\omega\}$ in it. $\{x_n:n<\omega\}$ is a sequence in
$U$ with no subsequences converging  in $\omega^*\times \{n\}$
and so a fortiori in $D$. 

We finish by giving a complete answer to Question $\ref{q4}$. Given a space $X$, the Pixley-Roy topology on $X$ is the space $\mathcal{F}[X]=[X]^{<\omega}$ equipped with the topology generated by sets of the form $[F,U]=\{G \in \mathcal{F}(X): F \subset G \subset U\}$, where $F$ is a finite subset of $X$ and $U$ is an open subset of $X$.

The authors of \cite{messico} proved that if $X$ is an SHD space whose every subset is closed and discrete (this hypothesis is verified, in particular if $X$ is a $P$-space), then $\mathcal{F}[X]$ is also SHD, and asked whether this is true in general.

\begin{theorem}
Let $X$ be any SHD space. Then $\mathcal{F}[X]$ is also SHD.
\end{theorem}

\begin{proof}
Let $\mathcal{U}$ be a countable sequence of non-empty open subsets of $\mathcal{F}[X]$. Without loss of generality we can assume that $\mathcal{U}$ is made up of basic open sets and thus we can enumerate $\mathcal{U}$ as $\{[F_n, U_n]: n < \omega \}$, where $F_n \in \mathcal{F}[X]$ and $U_n$ is a non-empty open subset of $X$. By the SHD property of $X$ we can pick a point $x_n \in U_n$, for every $n<\omega$ such that $\{x_n: n < \omega\}$ has no converging subsequence. Define $G_n=F_n \cup \{x_n\}$. Then $G_n \in [F_n, U_n]$, for every $n < \omega$. We claim that $\{G_n: n < \omega \}$ has no converging subsequence. Suppose that this is not the case and let $\{G_{n_k}: k < \omega \}$ be a subsequence converging to some point $G \in \mathcal{F}[X]$. That induces a subsequence $\{x_{n_k}: k < \omega\}$ of $\{x_n: n < \omega \}$ in the space $X$. Moreover, fix an enumeration $\{y_i: 1 \leq i \leq p\}$ of the set $G$. 

Since $S_0=\{x_{n_k}: k < \omega\}$ does not converge to $y_1$ then there are an infinite subset $S_1$ of $S_0$ and an open neighbourhood $U_1$ of $y_1$ such that $U_1 \cap S_1=\emptyset$. Now, since $S_1$ does not converge to $y_2$, there are an infinite subset $S_2$ of $S_1$ and an open neighbourhood $U_2$ of $y_2$ such that $U_2 \cap S_2=\emptyset$. Continuing in this way we can construct a decreasing sequence of infinite sets $\{S_i: 0 \leq i \leq p\}$ and a sequence of open sets $\{U_i: 1 \leq i \leq p\}$ such that $y_i \in U_i$ and $U_i \cap S_i=\emptyset$, for every $i \in \{1, \dots, p\}$. 

Notice that $U=\bigcup \{U_i: 1 \leq i \leq p\}$ is an open set which contains $G$ and is disjoint from $S_p$. It follows that the set $[G, U]$ is an open neighbourhood of $G$ in the Pixley-Roy topology which does not contain a tail of the sequence $\{G_{n_k}: k < \omega \}$ and that is a contradiction.
\end{proof}

\section{Acknowledgements}

Both authors were partially supported by the GNSAGA group of INdAM. In addition to that, the first named author was supported by a grant from ``Progetto PIACERI, linea intervento 2" of the University of Catania and the second-named author was supported by a grant from  the ``Fondo Finalizzato alla Ricerca di Ateneo'' (FFR 2023) of the University of Palermo.

\end{document}